\documentclass{amsart}

\usepackage{amsthm}
\usepackage{enumerate}
\usepackage{color}
\usepackage{float}
\usepackage{bbm}
\usepackage{amsmath}
\usepackage{comment}
\usepackage{listings}
\usepackage{ulem}
\usepackage[dvipsnames]{xcolor}
\usepackage{fourier}
\usepackage{hyperref}
\hypersetup{
	colorlinks,
	citecolor=blue,
	filecolor=blue,
	linkcolor=blue,
	urlcolor=blue
}
\allowdisplaybreaks

\newtheorem{theorem}{Theorem}[section]
\newtheorem{proposition}[theorem]{Proposition}
\newtheorem{lemma}[theorem]{Lemma}

\newtheorem{corollary}[theorem]{Corollary}
\newtheorem{definition}[theorem]{Definition}

\theoremstyle{definition}

\newcommand{\R}{\mathbb{R}}

\newcommand{\N}{\mathbb{N}}

\newcommand{\T}{\mathcal{T}}
\newcommand {\D} {\mathbb D}

\newcommand{\C}{\mathbb{C}}

\newcommand{\diag}{\textnormal{diag}}

\newcommand{\F}{\mathbb{F}}

\newcommand{\Col}{\textnormal{Col}}
\newcommand{\Row}{\textnormal{Row}}

\newcommand{\rank}{{\rm rank\,}}

\newcommand{\eps}{\varepsilon}

\begin{document}

\title{Generalized Pseudoskeleton Decompositions}

\author{Keaton Hamm}
\address{Department of Mathematics, University of Texas at Arlington, Arlington, TX 76019 USA}
\email{keaton.hamm@uta.edu}


\begin{abstract}
We characterize some variations of pseudoskeleton (also called CUR) decompositions for matrices and tensors over arbitrary fields. These characterizations extend previous results to arbitrary fields and to decompositions which use generalized inverses of the constituent matrices, in contrast to Moore--Penrose pseudoinverses in prior works which are specific to real or complex valued matrices, and are significantly more structured.
\end{abstract}

\maketitle

\section{Introduction}

Pseudoskeleton decompositions were introduced in their modern form by Gore\u\i{}nov et al.~\cite{Goreinov,Goreinov2,Goreinov3}, though their origins go back at least to Penrose \cite{Penrose56}. A more complete history and expository treatment can be found in \cite{hamm2020perspectives} (see also \cite{strang2022lu}). At their core, pseudoskeleton decompositions of matrices are decompositions of the form $A = CXR$ where $C$ and $R$ are column and row submatrices of $A$, respectively, and $X$ is some mixing matrix. While an interesting piece of linear algebra in their own right, such decompositions have also proven useful in various applications from sketching massive data matrices \cite{drineas2006fast,mahoney2009cur,voronin2017efficient}, accelerating nonconvex Robust PCA algorithms \cite{cai2020rapid,cai2021robust}, estimating massive kernel matrices \cite{GittensMahoney,williams2001using}, and various applications including clustering \cite{AHKS}, mass spectronomy \cite{yang2015identifying}, and analyzing corporate social responsibility \cite{amor2020bias}. Pseudoskeleton decompositions also go by other names including cross matrix approximation \cite{ahmadi2021cross,CaiafaTensorCUR} and CUR decomposition \cite{drineas2006fast}.

There are a few primary reasons that various communities have explored pseudoskeleton decompositions. In randomized and computational linear algebra, it is used as a fast approximation to the singular value decomposition (SVD). In particular, in certain cases, one can sample $O(k\log k)$ columns and rows of a matrix and return a factorization $A\approx CXR$ which approximates the truncated SVD of order $k$ up to a relative error of $1+\eps$ \cite{drineas2008relative} (recall that the truncated SVD is the best rank-$k$ approximation to a matrix in any Schatten $p$-norm). Additionally, existence of these decompositions for low-rank matrices allows one to speed up algorithms by only viewing a few columns and rows of a massive data matrix; this technique has proved useful in reducing the complexity of nonconvex solvers for Robust PCA by a factor of $n$ (up to log factors) \cite{cai2021fast,cai2020rapid,hamm2022riemannian}.   In data science and applications, pseudoskeleton decompositions are sometimes preferred because they provide a representation of a data matrix that is interpretable (singular vectors are hard to interpret, but columns and rows of data have a known meaning) \cite{ahmadi2021cross,cai2021robust,mahoney2009cur}.

The purpose of this note is to characterize various types of pseudoskeleton decompositions for matrices and tensors over arbitrary fields. Typically the mixing matrix $X$ involves the Moore--Penrose pseudoinverse, which is only valid for real or complex valued matrices and tensors. Here, we form $X$ via generalized inverses, which are much more abundant and have less structure than pseudoinverses, and we call these generalized pseudoskeleton decompositions. We find that many existing characterizations for pseudoskeleton decompositions hold in greater generality for arbitrary fields and arbitrary generalized inverses forming $X$. These generalize the results of \cite{CHHNjmlr,hamm2020perspectives}. However, in the tensor case, we find some differences in the case of arbitrary fields.  Additionally, we show that any tensor decomposition that involves matrix pseudoskeleton decompositions can be characterized in a similar way to the matrix case; we illustrate this fact by providing a characterization of t-CUR decompositions which are formed by the t-product between tensors as in \cite{wang2017missing}.

Section \ref{SEC:Prelim} is devoted to notation and surveying existing results; Section \ref{SEC:Matrix} characterizes generalized pseudoskeleton decompositions for matrices over arbitrary fields, and illustrates some classes of restricted generalized inverses and how they effect the decompositions. Section \ref{SEC:Tensor} characterizes several variants of generalized pseudoskeleton decompositions for tensors.

\section{Preliminaries}\label{SEC:Prelim}

Throughout what follows, $\F$ is an arbitrary field unless otherwise specified. $\F^{m\times n}$ is the set of all $m\times n$ matrices over the field $\F$, and $\F^{d_1\times \cdots\times d_n}$ is the set of all $n$-mode tensors over $\F$. We use $[n]$ to denote the set $\{1,\dots,n\}$. Given index sets $I\subset[m]$ and $J\subset[n]$, we will use $A(I,:)$ to denote the $|I|\times n$ row submatrix of $A$ corresponding to selecting the rows indexed by $I$, $A(:,J)$ to denote the $m\times|J|$ column submatrix of $A$ corresponding to selecting the columns indexed by $J$, and $A(I,J)$ to denote the $|I|\times|J|$ submatrix of $A$ corresponding to the intersection of the former two submatrices. Similar notation will be used for subtensors in Section \ref{SEC:Tensor}.

\begin{definition}
Let $A\in\F^{m\times n}$, then $A^\sim\in\F^{n\times m}$ is a \emph{generalized inverse} of $A$ provided $AA^\sim A = A$.
\end{definition}

Note that any $A\in\F^{m\times n}$ may be written as $A = F\begin{bmatrix}I & 0\\ 0 & 0\end{bmatrix}G^{-1}$ for some invertible matrices $F\in\F^{m\times m}$ and $G\in\F^{n\times n}$, and all generalized inverses of $A$ have the form \begin{equation}\label{EQN:GeneralizedInverse}A^\sim = G\begin{bmatrix} I & X\\ Y & Z\end{bmatrix}F^{-1},\end{equation} where $F$ and $G$ are the same invertible matrices defining $A$, but $X,Y,$ and $Z$ are arbitrary \cite{matsaglia1974equalities}.

The most standard and useful generalized inverse in the case $\F=\R$ or $\C$ is the Moore--Penrose pseudoinverse defined as follows.

\begin{definition}\label{DEF:MP}
Let $\F = \R$ or $\C$.  Then for every $A\in\F^{m\times n}$ the \emph{Moore--Penrose pseudoinverse} is the unique matrix $A^\dagger\in\F^{n\times m}$ which satisfies all of the following:
\begin{enumerate}[1.]
\item\label{ITEM:MP1} $AA^\dagger A = A$
\item\label{ITEM:MP2} $A^\dagger AA^\dagger = A$
\item\label{ITEM:MP3} $(AA^\dagger)^* = AA^\dagger$
\item\label{ITEM:MP4} $(A^\dagger A)^* = A^\dagger A$.
\end{enumerate}
\end{definition}

The following is the characterization of matrix pseudoskeleton (CUR) decompositions for real or complex valued matrices.

\begin{theorem}[\cite{hamm2020perspectives}]\label{THM:HHCUR}
Let $\F$ be $\C$ or $\R$, and $A\in\F^{m\times n}$.  Let $I\subset[m]$, $J\subset[n]$, $C:=A(:,J)$, $R=A(I,:)$, and $U=A(I,J)$. The following are equivalent:
\begin{enumerate}[(i)]
\item $\rank(U) = \rank(A)$
\item $A = CU^\dagger R$
\item $A = CC^\dagger AR^\dagger R$
\item $A^\dagger = R^\dagger U C^\dagger$
\item $\rank(C)=\rank(R)=\rank(A)$
\item Suppose columns and rows are rearranged so that $A = \begin{bmatrix} U & B \\ D & E\end{bmatrix}$, and the generalized Schur complement of $A$ with respect to $U$ is defined by $A/U:=E-DU^\dagger B$. Then $A/U = 0$, $D(I-U^\dagger U)=0$ and $(I-UU^\dagger)B=0$.
\end{enumerate}
\end{theorem}

We are primarily interested in understanding how Theorem \ref{THM:HHCUR} generalizes to the case that the Moore--Penrose pseudoinverses are replaced by generalized inverses and $\F$ is allowed to be arbitrary.

\section{Characterization of Generalized Pseudoskeleton Decompositions}\label{SEC:Matrix}

Here we present a characerization analogous to Theorem \ref{THM:HHCUR} for arbitrary generalized inverses in arbitrary fields. The generalization to arbitrary fields does not require any extra work; however, we are able to show below that the generalized inverses of all matrices can be arbitrary.

\begin{theorem}\label{THM:GeneralizedCUR}
Let $\F$ be an arbitrary field and $A\in\F^{m\times n}$.  Let $I\subset[m]$, $J\subset[n]$, $C = A(:,J)$, $R=A(I,:)$, and $U=A(I,J)$.  The following are equivalent:
\begin{enumerate}[(i)]
\item\label{ITEM:Generalized:rankUA} $\rank(U) = \rank(A)$
\item\label{ITEM:Generalized:rankCRA} $\rank(C)=\rank(R)=\rank(A)$
\item\label{ITEM:Generalized:ACUR} $A = CU^\sim R$ for some generalized inverse $U^\sim$ of $U$
\item\label{ITEM:Generalized:ACURall} $A = CU^\sim R$ for all generalized inverses $U^\sim$ of $U$
\item\label{ITEM:Generalized:ACCARR} $A = CC^\sim AR^\sim R$ for some generalized inverses $C^\sim$ and $R^\sim$
\item\label{ITEM:Generalized:ACCARRall} $A = CC^\sim AR^\sim R$ for all generalized inverses $C^\sim$ and $R^\sim$
\item\label{ITEM:Generalized:RUC} For some $C^\sim$ and $R^\sim$, $R^\sim UC^\sim$ is a generalized inverse of $A$
\item\label{ITEM:Generalized:RUCall} For all $C^\sim$ and $R^\sim$, $R^\sim UC^\sim$ is a generalized inverse of $A$
\item\label{ITEM:Generalized:Schur} Suppose columns and rows are rearranged so that $A = \begin{bmatrix} U & B \\ D & E\end{bmatrix}$, and the generalized Schur complement of $A$ with respect to $U$ is defined by $A/U:=E-DU^\sim B$ for some $U^\sim$. Then $A/U = 0$, $D(I-U^\sim U)=0$ and $(I-UU^\sim)B=0$ for some $U^\sim$
\item\label{ITEM:Generalized:Schurall} The conclusion of \eqref{ITEM:Generalized:Schur} holds for all $U^\sim$ (the $U^\sim$ in the equality $A/U=0$ must be the same as one of the others).
\end{enumerate}
\end{theorem}

\subsection{Some useful lemmata}

Here we collect some useful facts about the matrices involved in Theorem \ref{THM:GeneralizedCUR}.

\begin{lemma}\label{LEM:URAC}
Suppose that $A\in\F^{m\times n}$, $C=A(:,J)$, $U=A(I,J)$, and $R=A(I,:)$ for some index sets $I\subset[m],J\subset[n]$. Then for every generalized inverse of $A$, \[U=RA^\sim C.\]
\end{lemma}
\begin{proof}
Following \eqref{EQN:GeneralizedInverse}, suppose that $A = F\begin{bmatrix}I & 0\\ 0 & 0\end{bmatrix}G^{-1}$, and $A^\sim = G\begin{bmatrix} I & X\\ Y & Z\end{bmatrix}F^{-1}$.  Additionally let $P_I$ and $P_J$ be the row and column selection matrices, respectively, such that $P_IA = R$, $AP_J=C$, and $P_IC=U=RP_J$.  Then we have
\begin{align*}
RA^\sim C & = P_IF\begin{bmatrix}I & 0\\ 0 & 0\end{bmatrix}G^{-1} G\begin{bmatrix} I & X\\ Y & Z\end{bmatrix}F^{-1} F\begin{bmatrix}I & 0\\ 0 & 0\end{bmatrix}G^{-1} P_J\\
& = P_IF\begin{bmatrix}I & 0\\ 0 & 0\end{bmatrix}G^{-1}P_J \\
& = P_IAP_J\\
& = U.
\end{align*}
\end{proof}

Note that Lemma \ref{LEM:URAC} does not require any of the equivalent conditions of Theorem \ref{THM:GeneralizedCUR}.  The next lemma is obtained from (for example) \cite[Theorem 3]{matsaglia1974equalities}.  

\begin{theorem}[\cite{matsaglia1974equalities}]\label{THM:Projectors}
Let $A\in\F^{m\times n}$. If an idempotent matrix $X (=X^2)$ is a projector onto $\Col(A)$, then $X = AA^\sim$ for some $A^\sim.$ Conversely, for all $A^\sim$, $AA^\sim$ is a projector onto $\Col(A)$. In particular, there exists a bijection from the set of generalized inverses of $A$ to the set of idempotent projectors onto $\Col(A).$
\end{theorem}

\begin{lemma}\label{LEM:Projections}
Suppose that $A, C, U,$ and $R$ are as in Theorem \ref{THM:GeneralizedCUR}, and that $\rank(U)=\rank(C)=\rank(R)=\rank(A)$. Then
\begin{enumerate}[(i)]
    \item $CC^\sim A = A$ for all $C^\sim$
    \item $AR^\sim R = A$ for all $R^\sim$
    \item $CU^\sim U = C$ for all $U^\sim$
    \item $UU^\sim R = R$ for all $U^\sim$
    \item $UC^\sim C = U$ for all $C^\sim$
    \item $RR^\sim U = U$ for all $R^\sim$.
\end{enumerate} 
\end{lemma}

\begin{proof}
The rank conditions imply that $\Col(C)=\Col(A)$, $\Row(R)=\Row(A)$, $\Row(U)=\Row(C)$, and $\Col(U)=\Col(R)$. Therefore, Theorem \ref{THM:Projectors} implies that $CC^\sim$ is a projector onto $\Col(A)$, hence $CC^\sim A = A$. The rest of the identities follow similarly.
\end{proof}

Note that the analogue of Lemma \ref{LEM:Projections} when Moore--Penrose pseudoinverses are used is stronger: one obtains $CC^\dagger = AA^\dagger$, and similar conclusions for the rest of the projections due to uniqueness of the pseudoinverse and the fact that the projectors in question are orthogonal. For generalized inverses, these matrices are projectors, but may be oblique. Indeed, if $A$ and $A^\sim$ are as in \eqref{EQN:GeneralizedInverse}, then $AA^\sim$ has the form
\[AA^\sim = F\begin{bmatrix} I & X\\ 0 & 0\end{bmatrix}F^{-1},\]
where $X$ can be arbitrary.  

\subsection{Proof of Theorem \ref{THM:GeneralizedCUR}}

With the lemmata above, we may now supply the full proof of the main result from the previous section.

\begin{proof}[Proof of Theorem \ref{THM:GeneralizedCUR}]
Note that the equivalence \eqref{ITEM:Generalized:rankUA} $\Leftrightarrow$ \eqref{ITEM:Generalized:rankCRA} follows from Theorem \ref{THM:HHCUR}, and clearly \eqref{ITEM:Generalized:ACURall}, \eqref{ITEM:Generalized:ACCARRall},\eqref{ITEM:Generalized:RUCall},\eqref{ITEM:Generalized:Schurall} $\Rightarrow$ \eqref{ITEM:Generalized:ACUR},\eqref{ITEM:Generalized:ACCARR},\eqref{ITEM:Generalized:RUC},\eqref{ITEM:Generalized:Schur}.

Note that most proofs of Theorem \ref{THM:HHCUR} begin with \eqref{ITEM:Generalized:rankUA} $\Rightarrow$ \eqref{ITEM:Generalized:ACURall}, and the same proof method works in our case. However, for variety, we take a different approach here.

\eqref{ITEM:Generalized:rankUA} $\Rightarrow$ \eqref{ITEM:Generalized:Schurall}: Since $\rank(U)=\rank(A)$, (8.5) of \cite{matsaglia1974equalities} implies that
\[\rank\left(\begin{bmatrix} 0 & (I-UU^\sim) B\\ D(I-U^\sim U) & E-DU^\sim B\end{bmatrix}\right) = 0,\] for every $U^\sim$ (as long as the $U^\sim$ in the bottom right is the same as one of the others), whence the claim.

\eqref{ITEM:Generalized:Schur} $\Rightarrow$ \eqref{ITEM:Generalized:ACURall}: Without loss of generality, suppose the columns and rows of $A$ are ordered so that $C = \begin{bmatrix} U\\D\end{bmatrix}$, $R = \begin{bmatrix}U & B\end{bmatrix}$.  Then 
\[CU^\sim R = \begin{bmatrix} U\\D\end{bmatrix}U^\sim\begin{bmatrix}U & B\end{bmatrix} = \begin{bmatrix}UU^\sim U & UU^\sim B\\ DU^\sim U & DU^\sim B\end{bmatrix},\]
where the top left entry reduces to $U$.  Thus, \[A - CU^\sim R = \begin{bmatrix}0 & (I-UU^\sim) B\\ D(I-U^\sim U) & E-DU^\sim B\end{bmatrix},\] each entry of which is 0 by assumption; hence, $A = CU^\sim R$.




$\eqref{ITEM:Generalized:ACUR} \Rightarrow \eqref{ITEM:Generalized:rankCRA}$: Note that $\rank(C),\rank(R)\leq\rank(A)=\rank(CU^\sim R)\leq\min\{\rank(C),\rank(R)\}$, hence equality holds.

At this point we have equivalence of \eqref{ITEM:Generalized:rankUA}--\eqref{ITEM:Generalized:ACURall} and \eqref{ITEM:Generalized:Schur}-\eqref{ITEM:Generalized:Schurall}.

$\eqref{ITEM:Generalized:RUC} \Rightarrow \eqref{ITEM:Generalized:rankUA}$: Let $C^\sim$ and $R^\sim$ be some generalized inverses of $C$ and $R$, respectively, such that the condition holds. As a submatrix of $A$, we must have $\rank(U)\leq\rank(A)$. On the other hand, by assumption, we have
\[A = AR^\sim UC^\sim A,\]
which implies that $\rank(A)\leq\rank(U)$, hence \eqref{ITEM:Generalized:rankUA} holds.

$\eqref{ITEM:Generalized:rankUA} \Rightarrow \eqref{ITEM:Generalized:RUCall}$: Let $R^\sim$ and $C^\sim$ be arbitrary generalized inverses of $R$ and $C$, respectively. The assumption $\rank(U)=\rank(A)$ implies Lemma \ref{LEM:Projections}, whose conclusion combined with Lemma \ref{LEM:URAC} yields
\begin{align*}AR^\sim UC^\sim A & = AR^\sim RA^\sim CC^\sim A\\
& = AA^\sim A\\
& = A.
 \end{align*}
Since $R^\sim$ and $C^\sim$ are arbitrary, \eqref{ITEM:Generalized:RUCall} holds.

$\eqref{ITEM:Generalized:ACCARR} \Rightarrow \eqref{ITEM:Generalized:rankCRA}$: Note that $\rank(C),\rank(R)\leq\rank(A) = \rank(CC^\sim AR^\sim R)\leq\min\{\rank(C),\rank(R)\}$, hence equality holds.

$\eqref{ITEM:Generalized:rankCRA} \Rightarrow \eqref{ITEM:Generalized:ACCARRall}$: Let $C^\sim$ and $R^\sim$ be arbitary generalized inverses of $C$ and $R$, respectively.  Then Lemma \ref{LEM:Projections} and the definition of generalized inverse imply that 
\[CC^\sim AR^\sim R = AR^\sim R = A,\] and the proof is concluded.






\end{proof}

\subsection{Restricted classes of generalized inverses}

Theorem \ref{THM:GeneralizedCUR} is quite general in that it only assumes that the matrices involved are generalized inverses, which is a very mild requirement. However, there are more restricted classes of generalized interest that may be of interest; for instance, $\{i,j,k\}$--inverses.  These inverses are defined as follows with respect to the definition of the Moore--Penrose conditions.

\begin{definition}
Let $\{i,j,k\}\subset\{1,2,3,4\}$, with the possibility that any of $i$, $j$, or $k$ are empty.  If $A\in\F^{n\times m}$ (here $\F$ is arbitrary), then $A^\sim$ is an $\{i,j,k\}$--inverse of $A$ if $A^\sim$ satisfies conditions $i$, $j$, and $k$ of Definition \ref{DEF:MP}.
\end{definition}

Note that when $\F=\R$ or $\C$, then $\{i,j,k\}$--inverses always exist for any $i,j,k$; however, when $\F$ is a finite field, this may not be the case for all matrices $A$ and all $\{i,j,k\}$. Nonetheless, whenever they do exist, we are able to say something more about condition \eqref{ITEM:Generalized:RUC} of Theorem \ref{THM:GeneralizedCUR}.

First, we note that $\{1\}$--inverses and $\{1,2\}$--inverses always exist. The latter fact can be seen as follows: if $A^\sim$ and $\hat A^\sim$ are two generalized inverses of $A$, then $A^\sim A\hat{A}^\sim$ is a $\{1,2\}$--inverse of $A$.  Indeed, that it is a generalized inverse follows from
\[AA^\sim A\hat{A}^\sim A = AA^\sim A = A.\] Condition \ref{ITEM:MP2} is seen via
\[A^\sim A\hat{A}^\sim AA^\sim A\hat{A}^\sim = A^\sim AA^\sim A\hat{A}^\sim = A^\sim A\hat{A}^\sim.\]

\begin{theorem}\label{THM:ijk}
Invoke the notations of Theorem \ref{THM:GeneralizedCUR}, and assume that one of the equivalent conditions holds. Then,
\begin{enumerate}[(i)]
    \item For every $C^\sim$ and $R^\sim$, $R^\sim UC^\sim$ is a $\{1,2\}$--inverse of $A$
    \item If $C^\sim$ satisfies condition \ref{ITEM:MP3}, then $R^\sim UC^\sim$ does also for any $R^\sim$
    \item If $R^\sim$ satisfies condition \ref{ITEM:MP4}, then $R^\sim UC^\sim$ does also for any $C^\sim$
\end{enumerate}

\end{theorem}

\begin{proof}
Since one of the conditions of Theorem \ref{THM:GeneralizedCUR} holds, $R^\sim UC^\sim$ satisfies condition \ref{ITEM:MP1} by Theorem \ref{THM:GeneralizedCUR}\eqref{ITEM:Generalized:RUC}.  

Since we also have $A=CU^\sim R$ by Theorem \ref{THM:GeneralizedCUR}\eqref{ITEM:Generalized:ACUR}, applying Lemma \ref{LEM:Projections} yields
\begin{align*}R^\sim UC^\sim A R^\sim UC^\sim & = R^\sim UC^\sim CU^\sim RR^\sim UC^\sim\\
& = R^\sim UU^\sim UC^\sim\\
& = R^\sim UC^\sim.
\end{align*}
Hence, $R^\sim UC^\sim$ satisfies the second Moore--Penrose condition for every $C^\sim$ and $R^\sim$.

Suppose that $C^\sim$ satisfies condition \ref{ITEM:MP3}. Then again utilizing $A=CU^\sim R$ and Lemma \ref{LEM:Projections},
\begin{align*}
(AR^\sim UC^\sim)^* & = (CU^\sim RR^\sim UC^\sim)^*\\
& = (CU^\sim UC^\sim)^*\\
& = (CC^\sim)^*\\
& = CC^\sim,
\end{align*}
where the last equality is condition \ref{ITEM:MP3}.  Thus, $R^\sim UC^\sim$ satisfies condition \ref{ITEM:MP3}.

Suppose that $R^\sim$ satisfies condition \ref{ITEM:MP4}. Then similar to above,
\begin{align*}
(R^\sim UC^\sim A)^* & = (R^\sim UC^\sim CU^\sim R)^*\\
& = (R^\sim UU^\sim R)^*\\
& = (R^\sim R)^*\\
& = R^\sim R,
\end{align*}
where the last equality is condition \ref{ITEM:MP4}.  Thus, $R^\sim UC^\sim$ satisfies condition \ref{ITEM:MP4}.

\end{proof}

\subsection{Drazin Inverses}

Note that pseudoskeleton decompositions need not be valid for Drazin inverses, which are defined only for square matrices. A matrix $A^D\in\F^{n\times n}$ is a Drazin inverse of $A\in\F^{n\times n}$ if $A^DAA^D = A^D$, $AA^D = A^DA$, and $A^{k+1}A^D = A^k$ for some $k\in\N$ \cite{drazin1958pseudo}. Drazin inverses need not satisfy $AA^DA=A$ (i.e., need not be generalized inverses).  Consequently, the implication $\eqref{ITEM:Generalized:rankUA}\Rightarrow\eqref{ITEM:Generalized:ACUR}$ of Theorem \ref{THM:GeneralizedCUR} need not hold when $C=R=A$ and  $U^\sim=A^\sim$ is replaced with $U^D=A^D$.

\section{Generalized Tensor CUR Decompositions}\label{SEC:Tensor}

We now turn our attention to generalized tensor pseudoskeleton decompositions.  Interestingly, for tensors, there are more than one generalization of pseudoskeleton decompositions. Here, we first consider tensor decompositions in modewise product form similar to Tucker decompositions \cite{tucker1966}, which are of the form $\mathcal{T} = \mathcal{X}\times_1 W_1\times_2\dots\times_n W_n,$ where $\mathcal{X}$ is some core tensor, and the $W_i$ are so-called factor matrices. When restrictions are placed on the factor matrices $W_i$, other names may be applied; for instance, when $\F=\R$ or $\C$ and the $W_i$ are assumed to be orthogonal, then this is often called the Higher-Order SVD (HOSVD). For a thorough review of tensors and their decompositions, see \cite{kolda2009tensor} (see also \cite{ZOIA2018}).

\subsection{Tensor preliminaries}

Before proceeding with our discussion of generalized tensor pseudoskeleton decompositions, we pause to (briefly) collect some notation and basics of tensors. Tensors will be represented in calligraphic font, and an $n$-mode tensor over a field $\F$ is one with entries in $\F^{d_1\times\cdots\times d_n}$. Tensors can be unfolded along any mode ($i\in[n]$) into a matrix denoted $\mathcal{T}_{(i)}\in\F^{d_i\times \prod_{j\neq i}d_j}$. Tensors can be multiplied by matrices along a given mode ($i\in[n]$) and one has the equivalence \cite{kolda2009tensor}
\[(\mathcal{T}\times_i Y)_{(i)} = Y\mathcal{T}_{(i)},\]
where $\mathcal{T}\in\F^{d_1\times\cdots\times d_n}$ and $Y\in\F^{m\times d_i}$. The product $\mathcal{T}\times_i Y$ is a tensor of size $d_1\times \cdots \times d_{i-1}\times m\times d_{i+1}\times\cdots\times d_n$. One can perform successive modewise multiplication, and we note that $(\mathcal{T}\times_i Y)\times_j Z = (\mathcal{T}\times_j Z)\times_i Y$ \cite{ZOIA2018}, so we will use $\mathcal{T} = \mathcal{X}\times_{i=1}^n W_i$ to denote Tucker decompositions and similar successive modewise products.

The Kronecker product of matrices $A\in\F^{m\times n}$ and $B\in\F^{t\times s}$ is the matrix $A\otimes B\in\F^{mt\times ns}$ given by
\[A\otimes B := \begin{bmatrix} A_{11}B & \cdots & A_{1n}B\\ \vdots & \ddots & \vdots\\ A_{m1}B & \cdots & A_{mn}B \end{bmatrix}.\]

The notion of rank of a tensor is not uniquely defined for $n\geq3$. Here, we will only make use of multilinear rank. A tensor $\T\in{d_1\times\cdots\times d_n}$ is said to have multilinear rank $(r_1,\dots,r_n)$ if $\rank(\mathcal{T}_{(i)})=r_i$ for all $i=1,\dots,n$. Some works call this the Tucker rank of $\T$.

In \cite{CHHNjmlr}, two tensor decompositions are discussed -- Fiber and Chidori CUR decompositions. Chidori decompositions are the most direct extension of matrix CUR decompositions in that, given an $n$--mode tensor $\mathcal{T}\in\F^{d_1\times\dots\times d_n}$, one samples indices $I_j\subset[d_j]$ to form a core subtensor $\mathcal{R} = \mathcal{T}(I_1,\cdots,I_n)$. One then extrudes this subtensor out in all directions but one to yield the remaining constituent subtensors for the reconstruction. That is, the subtensors used in the reconstruction of $\mathcal{T}$ are $\mathcal{R}$ (taking the place of $U$ in the matrix case) and $\mathcal{C}_i = \mathcal{T}(I_1,\cdots,I_{i-1},:,I_{i+1},\cdots,I_n)$ (taking the place of the column submatrices). We write this in matrix form as $C_i = \mathcal{T}_{(i)}(:,\otimes_{j\neq i}I_j)$, where $\otimes_{j\neq i}$ denotes the mapping of the indices when unfolding $\mathcal{C}_i$ along the $i$-th mode. We also will utilize $U_i=C_i(I_i,:)$, which is similar to the row submatrix $R$ in the matrix setting.



For fiber CUR decompositions, one chooses $I_j\subset[d_j]$ and $J_i\subset \left[\prod_{j\neq i}d_j\right]$ and sets $\mathcal{R} = \mathcal{T}(I_1,\cdots,I_n)$ as before, but $C_i = \mathcal{T}_{(i)}(:,J_i)$ and $U_i = C_i(I_i,:)$. The main difference here is that in contrast with chidori CUR decompositions, one is allowed to choose fibers in each direction that may not intersect with the core subtensor $\mathcal{R}$ (i.e., the indices $J_i$ need not be derived from the indices $\{I_j\}_{j=1}^n$).

In either case, the tensor reconstruction (if possible) is 
\[\mathcal{T} = \mathcal{R}\times_1(C_1U_1^\dagger)\times_2\cdots\times_n(C_nU_n^\dagger).\]
For convenience, we will sometimes simplify this notation in the sequel to $\mathcal{R}\times_{i=1}^n(C_iU_i^\dagger)$. 

For a more extended description of tensor pseudoskeleton decompositions and similar approaches as well as algorithmic considerations, the reader is invited to consult \cite{ahmadi2021cross}. For figures describing the fiber and chidori CUR decompositions, see \cite{CHHNjmlr}.

\subsection{Characterizations for generalized fiber and chidori decompositions}

Here, we present two characterizations of generalized fiber and chidori decompositions described above to the case of generalized inverses and arbitrary fields. The fiber decomposition characterization is the same as it is for the real and complex case but for generalized inverses. However, we are unable to prove that one of the conditions for chidori decompositions is equivalent except in the case $\F=\R$ or $\C$ because the standard proof of existence of Tucker decompositions holds only for these fields. This brings up an interesting question: can the decomposition in \eqref{EQN:GeneralizedInverse} be extended to a Tucker-like decomposition for tensors over arbitrary fields? At present we do not have an answer to this, so we proceed with the characterizations below.

\begin{theorem}\label{THM:Fiber}
Let $\mathcal{T}\in\mathbb{F}^{d_1\times\cdots\times d_n}$ have multilinear rank $(r_1,\dots,r_n)$. Let $I_i\subset [d_i]$ and $J_i\subset[\prod_{j\neq i}d_j]$. Set $\mathcal{R}=\mathcal{T}(I_1,\cdots,I_n)$, $C_i=\mathcal{T}_{(i)}(:,J_i)$ and $U_i=C_i(I_i,:)$. 
Then the following statements are equivalent
\begin{enumerate}[(i)]
    \item\label{ITEM:FiberRankU} $\rank(U_i)=r_i$ for all $i=1,\dots,n$
    \item\label{ITEM:FiberDecomp} $\mathcal{T}=\mathcal{R}\times_{i=1}^nC_iU_i^\sim,$ for some $U^\sim_1,\dots,U^\sim_n$
    \item\label{ITEM:FiberDecompall} $\mathcal{T}=\mathcal{R}\times_{i=1}^nC_iU_i^\sim,$ for every $U^\sim_1,\dots,U^\sim_n$
    \item\label{ITEM:FiberMultiRank} $\rank(C_i)=r_i$ for all $i$ and the multilinear rank of $\mathcal{R}=(r_1,\dots,r_n)$.
\end{enumerate}
Moreover, if any of the equivalent statements above hold, then $\mathcal{T}=\mathcal{T}\times_{i=1}^n (C_iC_i^\sim)$ for every $C_i^\sim$. 
\end{theorem}

\begin{theorem}\label{THM:Chidori}
Let $\mathcal{T}\in\mathbb{F}^{d_1\times\cdots\times d_n}$ have multilinear rank $(r_1,\dots,r_n)$. Let $I_i\subset [d_i]$. Set $\mathcal{R}=\mathcal{T}(I_1,\cdots,I_n)$, $C_i=\mathcal{T}_{(i)}(:,\otimes_{j\neq i}I_j)$, and $U_i=C_i(I_i,:) = \mathcal{R}_{(i)}$. Then the following are equivalent:
\begin{enumerate}[(i)]
    \item\label{ITEM:ChidoriRankU} $\rank(U_i)=r_i$ for all $i=1,\dots,n$
    \item\label{ITEM:ChidoriDecomp} $\mathcal{T}=\mathcal{R} \times_{i=1}^n C_iU_i^\sim$ for some $U^\sim_1,\dots,U^\sim_n$
    \item\label{ITEM:ChidoriDecompall} $\mathcal{T}=\mathcal{R} \times_{i=1}^n C_iU_i^\sim$ for every $U^\sim_1,\dots,U^\sim_n$
    \item\label{ITEM:ChidoriRankR} the multilinear rank of $\mathcal{R}$ is $(r_1,\dots,r_n)$.
\end{enumerate}
If any of the equivalent statements above hold, then $\mathcal{T}=\mathcal{T}\times_{i=1}^n (C_iC_i^\sim)$ for every $C_i^\sim$.

Moreover, if any of the equivalent statements above hold, then
\begin{enumerate}[(i)]
\addtocounter{enumi}{4}
\item\label{ITEM:ChidoriRankTi}   $\rank(\mathcal{T}_{(i)}(I_i,:))=r_i$ for all $i=1,\dots,n$.
\end{enumerate}
If $\mathbb{F} = \R$ or $\C$, then \eqref{ITEM:ChidoriRankTi} is equivalent to the conditions above.
\end{theorem}

\begin{proof}[Proof of Theorem \ref{THM:Fiber}]

\eqref{ITEM:FiberRankU}$\Rightarrow$\eqref{ITEM:FiberDecompall}:
The proof will follow from successive applications of the matrix case \eqref{ITEM:Generalized:rankUA}$\Rightarrow$\eqref{ITEM:Generalized:ACURall}.  

First, note that by the matrix case, \begin{equation}\label{EQN:TensorModeCUR}\mathcal{T}_{(i)} = C_iU_i^\sim \mathcal{T}_{(i)}(I_i,:) = \mathcal{T}(:,\cdots,:,I_i,:,\cdots,:)\times_i C_iU_i^\sim,\quad \textnormal{for all } i.\end{equation}  In particular, begin with $i=1$ and set $\mathcal{R}_1:= \mathcal{T}(I_1,:,\cdots,:)$ and we have \begin{equation}\label{EQN:TensorCURStep1}\mathcal{T} = \mathcal{R}_1\times_1 C_1U_1^\sim.\end{equation}

Next, we have
$(\mathcal{R}_1)_{(2)} = \left(\mathcal{T}(I_1,:,\cdots,:)\right)_{(2)} = \mathcal{T}_{(2)}(:,I_1\otimes_{j=3}^n [d_j])$.  Letting $P_{1}$ be the column selection matrix which chooses the columns indexed by $I_1\otimes_{j=3}^n[d_j]$ from $\mathcal{T}_{(2)}$, then an application of \eqref{EQN:TensorModeCUR} yields
\[(\mathcal{R}_1)_{(2)} = \mathcal{T}_{(2)}P_1 = C_2U_2^\sim \mathcal{T}_{(2)}(I_2,:)P_1\]
implying
\[\mathcal{R}_1 = \mathcal{T}(I_1,I_2,:,\cdots)\times_2 C_2U_2^\sim=: \mathcal{R}_2\times_2 C_2U_2^\sim.\] 
Therefore, $\mathcal{T} = \mathcal{R}_2\times_2 C_2U_2^\sim\times_1 C_1U_1^\sim.$

Iterating this procedure, we find a sequence of subtensors $\mathcal{R}_2,\dots,\mathcal{R}_{n}$ for which $\mathcal{R}_k = \mathcal{T}(I_1,\cdots,I_{k+1},:,\cdots)\times_{k+1}C_{k+1}U_{k+1}^\sim=:\mathcal{R}_{k+1}\times_{k+1}C_{k+1}U_{k+1}^\sim$ for all $k = 1,\dots,n-1$.  Combining this with \eqref{EQN:TensorCURStep1} yields

\[\mathcal{T} = \mathcal{T}(I_1,\cdots,I_n)\times_{i=1}^nC_iU_i^\sim = \mathcal{R}\times_{i=1}^nC_iU_i^\sim.\]

$\eqref{ITEM:FiberDecomp}\Rightarrow\eqref{ITEM:FiberMultiRank}$: First, note that $C_i$ is a submatrix of $\mathcal{T}_{(i)}$ by definition, so $\rank(C_i)\leq\rank(\mathcal{T}_{(i)})=r_i$.  Next, since $\mathcal{R}_{(i)} = \mathcal{T}_{(i)}(I_i,\otimes_{j\neq i}I_j)$, $\rank(\mathcal{R}_{(i)})\leq r_i$.  Finally, \cite{kolda2009tensor,kolda2006multilinear} shows that condition $\eqref{ITEM:FiberDecomp}$ is equivalent to 
\[\mathcal{T}_{(i)} = C_iU_i^\sim\mathcal{R}_{(i)}(C_nU_n^\sim\otimes\dots\otimes C_{i+1}U_{i+1}^\sim\otimes C_{i-1}U_{i-1}^\sim\otimes\dots\otimes C_1U_1^\sim)^*, \]
whence $r_i=\rank(\T_{(i)})\leq\min\{\rank(C_i),\rank(\mathcal{R}_{(i)})\}\leq r_i$, hence $\eqref{ITEM:FiberMultiRank}$ holds.

$\eqref{ITEM:FiberMultiRank}\Rightarrow\eqref{ITEM:FiberRankU}:$ Note that $\mathcal{R}_{(i)}$ is a submatrix of $\mathcal{T}_{(i)}(I_i,:)$, so $r_i = \rank(\mathcal{R}_{(i)})\leq\rank(\mathcal{T}_{(i)}(I_i,:))\leq\rank(\mathcal{T}_{(i)})=r_i$.  Thus $\rank(\mathcal{T}_{(i)}(I_i,:)=\rank(C_i)=r_i=\rank(\mathcal{T}_{(i)})$, so $\rank(U_i)=r_i$ by the implication $\eqref{ITEM:Generalized:rankCRA}\Rightarrow\eqref{ITEM:Generalized:rankUA}$ of Theorem \ref{THM:GeneralizedCUR}.

The moreover statement follows from condition \eqref{ITEM:FiberDecompall} with $I_i=[d_i]$.
\end{proof}

\begin{proof}[Proof of Theorem \ref{THM:Chidori}]
The implications $\eqref{ITEM:ChidoriRankU}\Leftrightarrow\eqref{ITEM:ChidoriDecomp}\Leftrightarrow\eqref{ITEM:ChidoriDecompall}$ follow as special cases of Theorem \ref{THM:Fiber}. 

$\eqref{ITEM:ChidoriRankU}\Leftrightarrow\eqref{ITEM:ChidoriRankR}$: Note that 
 \[\rank(U_i) = \rank(\mathcal{R}_{(i)}) = r_i.\]

$\eqref{ITEM:ChidoriRankR}\Rightarrow \eqref{ITEM:ChidoriRankTi}$: By assumption,
\[r_i = \rank(\mathcal{R}_{(i)}) = \rank(\mathcal{T}_{(i)}(I_i,\otimes_{j\neq i}I_j)) \leq \rank(\mathcal{T}_{(i)}) = r_i,\]
hence equality holds, yielding the claim.


Finally, we prove that $\eqref{ITEM:ChidoriRankTi}\Rightarrow\eqref{ITEM:ChidoriRankU}$ in the case $\F=\R$ or $\C$.  In this case, we use the Tucker decomposition of $\mathcal{T}$  representation which takes the form $\mathcal{T} = \mathcal{X}\times_1 W_1\times_2\cdots \times_n W_n$ such that $\rank(\mathcal{X}_{(i)})=r_i$ and $W_i\in\R^{d_i\times r_i}$ has rank $r_i$.  Letting $P_{I_i}$ denote the row selection matrix that selects rows according to the index set $I_i$, notice that unfolding the Tucker decomposition of $\mathcal{T}$ along the $i$-th mode yields
\begin{equation}\label{EQN:RowT}P_{I_i}\mathcal{T}_{(i)} = P_{I_i}W_i\mathcal{X}_{(i)}(W_1\otimes\dots\otimes W_{i-1}\otimes W_{i+1}\otimes \dots \otimes W_n).\end{equation}
The left-hand side above is $\mathcal{T}_{(i)}(I_i,:)$ which has rank $r_i$ by assumption, so we see that $r_i = \rank(P_{I_i}\mathcal{T}_{(i)}) \leq \rank(W_i(I_i,:)) \leq \rank(W_i) = r_i,$ and so $\rank(W_i(I_i,:)) = r_i.$

Next, we notice that \[\mathcal{R} = \mathcal{T}(I_1,\dots,I_n) = \mathcal{T}\times_{i=1}^nP_{I_i} = \mathcal{X}\times_{i=1}^n W_i(I_i,:).\]
One can see this by applying the row selection matrix along each mode of $\mathcal{T}$ as in \eqref{EQN:RowT}.  Thus we have that $\rank(\mathcal{R}_{(i)})\leq\rank(W_i(I_i,:))=r_i.$  On the other hand, by repeated application of Sylvester's rank inequality, we have
\begin{align*}
    \rank(\mathcal{R}_{(i)}) & \geq \rank(W_i(I_i,:)\mathcal{X}_{(i)}(\otimes_{j\neq i}W_j)^*)\\
    & \geq \rank(W_i(I_i,:)\mathcal{X}_{(i)})+\rank(\otimes_{j\neq i}W_i)-\prod_{j\neq i} r_j\\
    & = r_i+\prod_{j\neq i}r_j - \prod_{j\neq i}r_j\\
    & = r_i.
\end{align*}
The first equality follows from the fact that $W_i(I_i,:)$ has full row rank and $\mathcal{X}_{(i)}$ has full column rank, both of which are $r_i$, and also that $\rank(A\otimes B) = \rank(A)\rank(B)$. Consequently, the multilinear rank of $\mathcal{R}$ is $(r_1,\dots,r_n)$, and the proof is complete.
\end{proof}

We note that the proof above of $\eqref{ITEM:ChidoriRankTi}\Rightarrow\eqref{ITEM:ChidoriRankU}$ requires the existence of a Tucker decomposition with orthogonal factor matrices. We only know this result for $\F=\R$ or $\C$, and do not know if it generalizes to arbitrary fields. The standard proof of existence of such a decomposition does not directly generalize to arbitrary fields as the factor matrices $W_i$ are required to be orthogonal, which may not be the case for arbitrary fields.

However, our proof above of $\eqref{ITEM:ChidoriRankU}\Leftrightarrow\eqref{ITEM:ChidoriRankR}$ does not go through the Tucker decomposition of $\mathcal{T}$ as does the proof of the similar fact for Moore--Penrose pseudoinverses in \cite{CHHNjmlr}, so our proof technique here allows for a more general result which holds for arbitrary fields and generalized inverses.

\subsection{Other tensor CUR decompositions}

There are other tensor variants of CUR decompositions, but most have been studied only for 3-mode tensors. Perhaps the most natural is that based on the t-SVD, which is formed via multiplying block circulant matrices obtained from slabs of 3-tensors. Here, we briefly illustrate how one can obtain characterizations of t-CUR decompositions for 3-mode tensors from our previous analysis. However, we note that the Fiber and Chidori decompositions characterized above are more natural extensions of pseudoskeleton approximations to higher order tensors, as the t-SVD construction is essentially limited to 3-mode tensors.  Our treatment here follows \cite{wang2017missing} (see also \cite{chen2022tensor}). Since the analysis depends only on using matrix CUR decompositions to extend the results, we are deliberately brief in our proof of the extension to generalized inverses.

In this section, we restrict to $\F=\R$ or $\C$. Given $\mathcal{T}\in\F^{m\times n\times \ell}$, one can define the associated block circulant matrix obtained from the mode-3 slabs of $\mathcal{T}$, which we denote $\mathcal{T}_{i}$ ($=\mathcal{T}(:,:,i)$) for simplicity, via
\[\textnormal{bcirc}(\mathcal{T}):=\begin{bmatrix} \mathcal{T}_1 & \mathcal{T}_{\ell} & \cdots & \mathcal{T}_2\\
\mathcal{T}_2 & \mathcal{T}_1 & \cdots & \mathcal{T}_3\\
\vdots & \vdots & \ddots & \vdots\\
\mathcal{T}_{\ell} & \mathcal{T}_{\ell-1} & \cdots & \mathcal{T}_1\end{bmatrix} \in \F^{m\ell\times n\ell}.\]
For the purpose of this section, we will utilize a slight modification of the unfolding of a matrix along its second mode. Rather than stacking the mode-3 slabs horizontally, we will stack them vertically, so that we define
\[\mathcal{T}_{[2]}:= \begin{bmatrix} \mathcal{T}_1\\ \vdots \\ \mathcal{T}_{\ell}\end{bmatrix}\in\F^{m\ell\times n}.\] 

Then the t-product of tensors $\mathcal{T}\in\F^{m\times n\times \ell}$ and $\mathcal{S}\in\F^{n\times k\times \ell}$ is the tensor of dimension $m\times k\times \ell$ obtained via circular convolution. That is, if $F_\ell$ is the $\ell\times\ell$ DFT matrix, then the circular DFT of $\mathcal{T}$ is the block diagonal matrix
\begin{equation}\label{EQN:That}\widehat{\T}:= (F_\ell\otimes I_{m\times m})\textnormal{bcirc}(\mathcal{T})(F_\ell^*\otimes I_{n\times n}) = \diag\left(\widehat\T_1,\dots,\widehat\T_{\ell}\right).\end{equation}
Then the t-product of $\T$ with $\mathcal{S}$ is denoted $\T\ast\mathcal{S}$ and given by
\[(\T\ast\mathcal{S})_{[2]} := \widehat\T \mathcal{S}_{[2]}.\]

The following lemma is well-known and straightforward to prove via basic properties of the DFT matrix.
\begin{lemma}\label{LEM:FTProduct}
Let $\F=\R$ or $\C$. Let $\mathcal{T}\in\F^{m\times n\times\ell}$ and $\mathcal{S}\in\F^{n\times k\times\ell}$. Then $\T\ast\mathcal{S} = \mathcal{H}$ if and only if $\widehat\T\widehat{\mathcal{S}}=\widehat{\mathcal{H}}$.
\end{lemma}

From this, we directly obtain the following straightforward proposition.
\begin{proposition}\label{PROP:CWRtprod}
Let $\F=\R$ or $\C$. Let $\mathcal{T}\in\F^{m\times n\times\ell}$, and let $\mathcal{C} = \T(:,J,:)\in\F^{m\times|J|\times\ell}$ and $\mathcal{R}=\T(I,:,:)\in\F^{|I|\times n\times\ell}$. Then for $\mathcal{W}\in\F^{|J|\times|I|\times\ell}$, we have $\T = \mathcal{C}\ast\mathcal{W}\ast\mathcal{R}$ if and only if $\widehat{\T}_i =\widehat{\mathcal{C}}_i\widehat{\mathcal{W}}_i\widehat{\mathcal{R}}_i$ for all $i.$
\end{proposition}

\begin{proof}
Combine Lemma \ref{LEM:FTProduct} with the representation \eqref{EQN:That}.
\end{proof}

Now let us define generalized inverses for tensors. Given $\mathcal{T}\in\F^{m\times n\times \ell}$, we define $\mathcal{T}^\sim$ via its Fourier transform as
\[\widehat{\mathcal{T}^\sim} = \diag(\widehat{\T}_1^\sim,\dots,\widehat{\T}_\ell^\sim)\]
where $\widehat{\mathcal{T}}_i^\sim$ is any generalized inverse of $\widehat\T_i$. In the terminology of Behera et al.~\cite{behera2022computation}, $\T^\sim$ is an \textit{inner inverse} of $\T$, i.e., it satisfies $\T\ast\T^\sim\ast\T = \T$. This can be seen from the corresponding matrix definition of generalized inverse applied along with Lemma \ref{LEM:FTProduct}.  With these results in hand, we can directly prove the following simply by appealing to the generalized matrix pseudoskeleton decomposition theorem (Theorem \ref{THM:GeneralizedCUR}).

\begin{corollary}
Let $\F=\R$ or $\C$. Let $\mathcal{T}\in\F^{m\times n\times\ell}$ have multilinear rank $(r_1,r_2,r_3)$, and let $\mathcal{C} = \T(:,J,:)\in\F^{m\times|J|\times\ell}$ and $\mathcal{R}=\T(I,:,:)\in\F^{|I|\times n\times\ell}$, and $\mathcal{U}=\mathcal{T}(I,J,:)$.  The following are equivalent:
\begin{enumerate}[(i)]
    \item $\rank(\mathcal{U}_i)=r_i$ for all $i$
    \item $\rank(\mathcal{C}_i)=\rank(\mathcal{R}_i)=r_i$ for all $i$
    \item $\mathcal{T}=\mathcal{C}\ast\mathcal{U}^\sim\ast\mathcal{R}$ for some generalized inverse $\mathcal{U}^\sim$
    \item $\mathcal{T}=\mathcal{C}\ast\mathcal{U}^\sim\ast\mathcal{R}$ for all generalized inverse $\mathcal{U}^\sim$
    \item $\T=\mathcal{C}\ast\mathcal{C}^\sim\ast\mathcal{T}\ast\mathcal{R}^\sim\ast\mathcal{R}$ for some generalized inverses $\mathcal{C}^\sim$ and $\mathcal{R}^\sim$
    \item $\T=\mathcal{C}\ast\mathcal{C}^\sim\ast\mathcal{T}\ast\mathcal{R}^\sim\ast\mathcal{R}$ for all generalized inverses $\mathcal{C}^\sim$ and $\mathcal{R}^\sim$
    \item For some $\mathcal{C}^\sim$ and $\mathcal{R}^\sim$, $\mathcal{R}^\sim\ast\mathcal{U}\ast\mathcal{C}^\sim$ is a generalized inverse of $\T$
    \item For all $\mathcal{C}^\sim$ and $\mathcal{R}^\sim$, $\mathcal{R}^\sim\ast\mathcal{U}\ast\mathcal{C}^\sim$ is a generalized inverse of $\T$.
\end{enumerate}
\end{corollary}

\begin{proof}
Combine Proposition \ref{PROP:CWRtprod} with Theorem \ref{THM:GeneralizedCUR}.
\end{proof}

We note that one can also obtain the analogue of Theorem \ref{THM:ijk} directly via the techniques above. The decomposition $\mathcal{T} = \mathcal{C}\ast\mathcal{U}^\dagger\ast\mathcal{R}$ was called the t-CUR decomposition in \cite{wang2017missing}.

We conclude this section with a general remark that any tensor factorization that is obtained directly from matrix CUR decompositions can be shown to work in the generalized inverse case by appealing to Theorem \ref{THM:GeneralizedCUR}, so while we illustrated this here on fiber and chidori CUR and t-CUR, the technique generalizes in a straightforward manner.

\section*{Acknowledgements}

The latter part of the writing of this manuscript was done while the author was a visitor to the Fields Institute for Research in Mathematical Sciences. The author was partially supported by the Fields Institute while there. The author thanks the institute for their hospitality. The author also thanks Longxiu Huang for helpful discussions related to the manuscript.

\bibliographystyle{plain}
\bibliography{references}

\end{document}